\documentclass[a4j, 12pt, draft]{amsart}
\usepackage{amsmath,amssymb}	
\usepackage{amsthm}
\usepackage{mathrsfs}	
\usepackage{bm}	

\allowdisplaybreaks[0]

\theoremstyle{plain} 
\newtheorem{theorem}{Theorem}
\newtheorem{lemma}{Lemma}
\newtheorem{corollary}{Corollary}
\newtheorem{conjecture}{Conjecture}

\newtheorem{proposition}{Proposition}

\newtheorem*{conjecture*}{Conjecture}
\newtheorem*{theorem*}{Theorem}

\theoremstyle{plain}

\theoremstyle{remark}

\theoremstyle{definition}
\newtheorem{definition}{Definition}
\newtheorem*{acknowledgment*}{Acknowledgments}

\numberwithin{equation}{section}

\DeclareMathOperator{\const}{const}

\newcommand{\RR}{\mathbb{R}}

\newcommand{\e}{\varepsilon}
\newcommand{\s}{\sigma}

\newcommand{\Lam}{\Lambda}

\newcommand{\us}{\underset}

\renewcommand{\a}{\alpha}
\renewcommand{\b}{\beta}
\renewcommand{\l}{\left}
\renewcommand{\r}{\right}
\renewcommand{\d}{\displaystyle}
\renewcommand{\Re}{\mathrm{Re}}
\renewcommand{\Im}{{\rm Im}}
\renewcommand{\epsilon}{\varepsilon}

\newcommand{\todaye}{\the\year/\the\month/\the\day}

\allowdisplaybreaks[0]

\title[The distribution of zeros and $\log{\zeta(s)}$]
{On the behavior of the logarithm\\ of the Riemann zeta-function}
\author[S. INOUE]{Sh\={o}ta Inoue}

\address{Graduate School of Mathematics, Nagoya University,
Furocho, Chikusaku, Nagoya 464-8602, Japan}
\email{m16006w@math.nagoya-u.ac.jp}


\begin{document}

\maketitle

\begin{abstract}
The purpose of the present paper is to reveal the relation between the behavior of the logarithm of the Riemann zeta-function 
$\log{\zeta(s)}$ and the distribution of zeros of the Riemann zeta-function.
We already know some examples for the relation by some previous works.
For example, Littlewood showed an upper bound of $\log{\zeta(1/2 + it)}$ by assuming the Riemann Hypothesis in 1924.
One of our results reveals that Littlewood's upper bound 
can be proved without assuming a hypothesis as strong as the Riemann Hypothesis.
\end{abstract}



\section{\textbf{Introduction}}


In the present paper, we discuss the behavior of the logarithm of the Riemann zeta-function $\log{\zeta(s)}$ 
under an assumption related to the distribution of zeros of the Riemann zeta-function.
As classical upper bounds of this function, we know well that
\begin{gather}
\Re\l( \log{\zeta\l( 1/2 + it \r)} \r) = \log{\l| \zeta\l( 1/2 + it \r) \r|} \leq C\log{t}, \label{UUBR} \\
S(t) := \frac{1}{\pi}\Im\l( \log{\zeta(1/2 + it)} \r) \ll \log{t}.	\label{UUBI}
\end{gather}
On the other hand, Littlewood \cite{LES} showed the estimates
\begin{align}	\label{LUB}
\log{\l|\zeta\l( 1/2 + it \r)\r|}
\leq \frac{C\log{t}}{\log{\log{t}}}, \quad
S(t) \ll \frac{\log{t}}{\log{\log{t}}}
\end{align}
with $C$ a positive constant under the Riemann Hypothesis.
About one hundred years have passed since the above estimates were shown, 
but still it is difficult to improve these estimates even today.
Of course, in view of Littlewood's upper bounds \eqref{LUB}, we believe that 
classical estimates \eqref{UUBR}, \eqref{UUBI} are not best possible.
Moreover, by the following conjecture, we expect that it is also possible to improve Littlewood's upper bounds \eqref{LUB}.
\begin{conjecture}[Farmer, Gonek, and Hughes \cite{FGHML}]
\begin{gather*}
\max_{t \in [0, T]}\log{\l| \zeta(1/2 + it) \r|} = \l( \frac{1}{\sqrt{2}} + o(1) \r)\sqrt{(\log{T})(\log{\log{T}})},\\
\limsup_{t \rightarrow +\infty}\frac{S(t)}{\sqrt{(\log{t})(\log{\log{t}})}} = \frac{1}{\pi\sqrt{2}}.
\end{gather*}
\end{conjecture}
Therefore, we would like to understand the behavior of $\log{\zeta(s)}$ more deeply to improve the above estimates.
From this perspective, as interesting works, there are some studies for the implicit constant 
in estimates \eqref{UUBR}, \eqref{UUBI}, and \eqref{LUB}.
For example, for sufficiently large $t$, Bourgain \cite{BER} showed the inequality 
$
\log{|\zeta(1/2 + it)|} \leq (13/84 + \e)\log{t}
$
, and 
Trudgian \cite{TSII} showed the inequality 
$
|S(t)| \leq (0.112 + \e)\log{t}
$.
Assuming the Riemann Hypothesis, Chandee and Soundararajan \cite{CSBR} showed the inequality 
$\log{|\zeta(1/2 + it)|} \leq (\frac{\log{2}}{2} + \e)\frac{\log{t}}{\log{\log{t}}}$, and 
Carneiro, Chandee, and Milinovich \cite{CCMB} showed 
$
|S(t)| \leq \l( \frac{1}{4} + \e \r)\frac{\log{t}}{\log{\log{t}}}.
$ 
Moreover, some interesting omega-results were also shown 
by Montgomery \cite{MEV}, Soundararajan \cite{SEV}, and Tsang \cite{KMTO}.

In the present paper, to understand the behavior of $\log{\zeta(s)}$ more deeply, 
we focus on the relation between this behavior and the distribution of zeros. 
Now, we already know some results of such a type.
For example, Backlund gave a statement equivalent to the Lindel\"of Hypothesis 
in terms of the distribution of zeros of $\zeta(s)$.
Let $N(\s, T, h)$ be the number of non-trivial zeros $\rho = \b + i\gamma$ of the Riemann zeta-function with 
$\b \geq \s$ and $T \leq \gamma \leq T + h$ counted with multiplicity. 
The Lindel\"of Hypothesis means that, for any fixed $\e > 0$, the estimate
\begin{align*}
\log{\l|\zeta\l( 1/2 + it \r)\r|} \leq \e\log{|t|}
\end{align*}
holds for $|t| \geq T(\e)$. 
Backlund \cite{BU} showed that the Lindel\"of Hypothesis is equivalent to the following statement:
\textit{for any fixed number $\e > 0$, 
\begin{align*}
N(\s, T, 1) = o(\log{T}) \quad (T \rightarrow +\infty)
\end{align*}
holds for $\s \geq \frac{1}{2} + \e$.}
Thanks to this equivalence, we can regard the Lindel\"of Hypothesis as a problem for the distribution of zeros of $\zeta(s)$ 
whose real parts are strictly greater than one-half.
In addition, Cram\'er \cite{CLS} showed the estimate
\begin{align}	\label{CUBS}
S(t) = o(\log{t}) \quad (T \rightarrow +\infty)
\end{align}
holds under the Lindel\"of Hypothesis.
Therefore, we may guess that there is a relation between the behavior of $\log{\zeta(s)}$ 
and the distribution of zeros of $\zeta(s)$.

The purpose of the present paper is to describe such a relation more clearly.
To achieve this purpose, the author introduces the following definition.


\begin{definition}[Short Interval Zero Density Condition]
Let $l(t), v(t)$ be nonnegative even functions weakly decreasing for sufficiently large $t$, 
and let $\Phi(t), \Psi(t)$ be even functions weakly increasing and greater than three for sufficiently large $t$.
Then, for an interval $I$ on $\RR$, consider the assertion ``the following estimate
\begin{align}
N(\s, T, l(T)) \leq l(T) v(T) (\log{T}) \Phi(T)^{1/2 - \s} 
\end{align}
holds for $T \in I$, $\s \geq \frac{1}{2} + \frac{1}{\Psi(T)}$". 
We call this assertion ``\textit{the Short Interval Zero Density Condition of length $l$, volume $v$, density $\Phi$, and domain $\Psi$ on $I$}".
In this article, we call the assumption ``SIZDC-$(l, v, \Phi, \Psi)$ on $I$" for brevity.
In addition, we may simply call the assertion ``SIZDC-$(l, v, \Phi, \Psi)$" in the case $I = \RR$.
\end{definition}


Now the author mentions some remarks for this definition.

First, the assertion SIZDC-$(\mathbf{1}, v, \mathbf{1}, \Psi)$ becomes an unconditional estimate 
for any positive valued function $\Psi$
if $v$ is a constant function whose value is sufficiently large. 
Here the function $\mathbf{1}$ indicates the identically on function.

Secondly, we can express the Riemann Hypothesis in terms of the SIZDC.
Actually, the Riemann Hypothesis is equivalent to that, 
for any functions $l$, $\Phi$, and $\Psi$, 
the SIZDC-$(l, \mathbf{0}, \Phi, \Psi)$ holds, 
where the function $\mathbf{0}$ indicates the identically zero function.

Thirdly, we can also express the Lindel\"of Hypothesis in terms of the SIZDC.
By Backlund's work, the Lindel\"of Hypothesis is equivalent to that, 
for any bounded function $\Psi$, there exists a function $v(t) = o(1)$ such that
the SIZDC-$(\mathbf{1}, v, \const, \Psi)$ holds. 
Here, the function $\mathbf{1}$ means identically one function, 
and the function $\const$ means some constant function. 

Therefore, this definition can inclusively deal various situations,
and the author believes that studies of $\log{\zeta(s)}$ under various situations 
are important to describe relations between the behavior of $\log{\zeta(s)}$ and the distribution
of zeros of the Riemann zeta-function.
In fact, the author gives a sufficient condition for Littlewood's bounds \eqref{LUB} 
that is weaker than the Riemann Hypothesis.

\section{\textbf{Notations}} \label{notation}

We define some notations in this section. 
Let $s = \s + it$ be a complex number with $\s, t$ real numbers.
Let $t \geq 14$, $x \geq 3$, and $0 < a < 1$, and put $\delta_x = (\log{x})^{-1}$ and $s_1 = \s_1 + it$ with $\s_1 = \frac{1}{2} + a + \delta_x$.
Let $\Lam$ be the von Mangoldt function. 
The modified von Mangoldt function $\Lam_x$ is defined by
\begin{align*}
\Lam_{x}(n) = \l\{
\begin{array}{cl}
\d{\Lam(n)}						& \text{if \; $n \leq x$,} \vspace{1mm}\\
\d{\Lam(n)\frac{\log(x^2 / n)}{\log{x}}}	& \text{if \; $x \leq n \leq x^2$,} \vspace{1mm}\\
0							& \text{otherwise.}
\end{array}
\r.
\end{align*}
We also define the set $A = A(x, t)$ by 
\begin{align}	\label{def_A}
A = \l\{ \b + i\gamma \; \middle| \; \zeta(\b + i\gamma) = 0, 
|t - \gamma| \leq \min\l\{\frac{t}{2}, \frac{x^{3(\b - \frac{1}{2})}}{\sqrt{\log{x}}}\r\} \r\},
\end{align}
and positive numbers $\s_A = \s_A(x, t)$ and $L = L(x, t)$ by
\begin{gather}
\label{def_s_M}
\s_A = \max\l\{ \b \mid \b + i\gamma \in A \r\},\\ 
\label{def_L}
L = \min\l\{\frac{t}{2}, \max\l\{ \frac{x^{3(\b - \frac{1}{2})}}{\sqrt{\log{x}}}
\middle| \b + i\gamma \in A \r\}\r\}.
\end{gather} 
Moreover, we define
\begin{align}	\label{def_tau}
\tau(a) = \l\{
\begin{array}{ll}
1	&	\text{if \; $a \leq \s_A$,} \vspace{1mm}\\
0	&	\text{if \; $a > \s_A$,}
\end{array}
\r.
\end{align}
\begin{align}	\label{def_F}
F_{a}(x, t) := \l( \frac{x}{\Phi(t/2)} \r)^{a} \times \sum_{k = 0}^{[(\s_A - a)\log{\Phi(t/2)}]}\l( \frac{x^2}{\Phi(t/2)} \r)^{(k + 1) / \log{\Phi(t/2)}},
\end{align}
\begin{align}	\label{def_G}
G_{a}(x, t) : = \tau(a)(l(t / 2) + \delta_x) v(t / 2) (\log{x}) (\log{t}),
\end{align}
\begin{align}	\label{def_Y}
&Y_{a}(\s, x, t) := \\ \nonumber
&\frac{x^{1/2 + a - \s}}{\log{x}} \l( \l| \sum_{n \leq x^{2}}\frac{\Lam_{x}(n)}{n^{s_1}} \r| + \log{t}\r) 
+ G_{a}(x, t) \times \Bigg\{x^{1/2 - \s}\frac{F_{a}(x, t)}{\log{x}} + \\ \nonumber
&\qquad \quad + \frac{x^{1/2 - \s}}{\log{x}} \l( 1 + \frac{\Phi(t / 2)^{-\delta_x}\log{x}}{\log{\Phi(t / 2)}} \r) \l(\frac{x}{\Phi(t / 2)}\r)^a
 + \frac{\Phi(t / 2)^{1/2 - \s + \delta_x}}{\log{\Phi(t / 2)}} \Bigg\},
\end{align}
and
\begin{align} \label{def_E}
E_{a}(x, t) 
= &\l| \sum_{n \leq x^2}\frac{\Lam_{x}(n)}{n^{s_1}} \r| + \log{t}
+ G_{a}(x, t) \times \\ \nonumber
& \quad \times \l( x^{-a}F_{a}(x, t) + \Phi(t / 2)^{-a}\l(1 + \frac{\log{x}}{\log{\Phi(t / 2)}} \Phi(t / 2)^{-\delta_x}\r) \r).
\end{align}

\section{\textbf{Results}}	\label{RSS}

The following assertion is the main theorem in the present paper.

\begin{theorem}	\label{main_thm}
Let $t \geq 14$ be not the ordinate of zeros of the Riemann zeta-function and $x$ be a number with $3 \leq x \leq \min\l\{e^{\Psi(t / 2)}, t^2 \r\}$, 
and put $\delta_x = (\log{x})^{-1}$ and $s_x = \s_x + it$ with $\s_x = \frac{1}{2} + 2\delta_x$.
Assume the SIZDC-$(l, v, \Phi, \Psi)$ on $[t - L, t + L]$ with $L$ defined by \eqref{def_L}. 
If $\s_x \leq \s \leq 2$, then we have
\begin{align*}
\log{\zeta(s)} 
= \sum_{|s - \rho| \leq \delta_x}\log\l|\frac{s - \rho}{\delta_x + i(t - \gamma)} \r|
+ \sum_{2 \leq n \leq x^2}\frac{\Lam_{x}(n)}{n^{s} \log{n}} + O\l( Y_{\delta_x}(\s, x, t) \r),
\end{align*}
where $Y_{a}(\s, x, t)$ is defined by \eqref{def_Y}.
If $\frac{1}{2} \leq \s \leq \s_x$, then we have
\begin{align}	\label{maineq_2}
\log{\zeta(s)} 
= &\sum_{|t - \gamma| \leq \delta_x}\log\l| \frac{s - \rho}{s_x - \rho} \r|
+\sum_{|s_x - \rho| \leq \delta_x}\log\l| \frac{s_x - \rho}{\delta_x + i(t - \gamma)} \r|\\ \nonumber
&+ \sum_{2 \leq n \leq x^2}\frac{\Lam_{x}(n)}{n^{s_x} \log{n}}
+ O\l(\delta_x E_{\delta_x}(x, t) \r),
\end{align}
where $E_{a}(x, t)$ is defined by \eqref{def_E}.
\end{theorem}

We can obtain some results on the behavior of $\log{\zeta(s)}$ 
by applying this theorem.
Actually, we give an important application as the following corollary, which immediately follows 
by applying Theorem \ref{main_thm} with $x = \l( \log{\frac{t}{2}} \r)^{\e_0 / 4}$.

\begin{corollary}	\label{LUBSI}
Let $\e_0$ be a small positive number and $t$ be a sufficiently large number, which
does not coincide with the ordinate of zeros of the Riemann zeta-function.
Assume SIZDC-$(l, \mathbf{1}, \Phi, \Psi)$ with 
$l(t) = \frac{1}{\log{\log{t}}}$, $\Phi(t) = (\log{t})^{\e_0}$ with $\e_0$ 
any fixed small positive constant, and 
$\Psi(t) = \e_0\log{\log{t}}$.
Then we have
\begin{align*}
\log{\zeta\l(\frac{1}{2} + it \r)} 
= \sum_{|\frac{1}{2} + it - \rho| \leq \frac{1}{\log{\log{t}}}}
\log\l| \frac{\frac{1}{2} + it - \rho}{\frac{1}{2} + \frac{8}{\e_0\log{\log{t}}} + it - \rho} \r|
 + O_{\e_0}\l( \frac{\log{t}}{\log{\log{t}}} \r).
\end{align*}
In particular, for any sufficiently large number $t$, we have
\begin{align*}
\log{\l|\zeta\l( \frac{1}{2} + it \r)\r|}
\leq \frac{C\log{t}}{\log{\log{t}}}, \quad
S(t) \ll_{\e_0} \frac{\log{t}}{\log{\log{t}}},
\end{align*}
where $C$ is a positive constant depending only on $\e_0$.
\end{corollary}

By this corollary, we can understand that 
Littlewood's upper bound \eqref{LUB} can be obtained  
without assuming the Riemann Hypothesis. 
Actually, we can rewrite the condition of Corollary \ref{LUBSI} into the following 
assertion, which is obviously weaker than the Riemann Hypothesis
: \textit{for any fixed small positive constant $\e_0$,
the estimate
\begin{align*}
N\l(\s, t, \frac{1}{\log{\log{t}}} \r)
\ll (\log{\log{t}})^{-1}(\log{t})^{1 - \e_0(\s - 1/2)}
\end{align*}
holds for any sufficiently large number $t$ and $\s \geq \frac{1}{2} + \frac{1}{\e_0\log{\log{t}}}$.}

\section{\textbf{Preliminaries for the proof of Theorem \ref{main_thm}}}

In this section, we prepare some lemmas. 
These lemmas are necessary to prove the following theorem, 
which is a generalization of Theorem \ref{main_thm}. 

\begin{theorem}	\label{Gen_thm}
Let $t \geq 14$ not the ordinate of zeros of the Riemann zeta-function, 
$3 \leq x \leq t^2$, and put $\delta_x = (\log{x})^{-1}$.
Assume the SIZDC-$(l, v, \Phi, \Psi)$ on $[t - L, t + L]$ with $L$ defined by \eqref{def_L}, 
and let $\frac{1}{\Psi(t / 2)} \leq a \leq 1$ and $s_{1} = \s_{1} + it$ 
with $\s_{1} = \frac{1}{2} + a + \delta_x$.
If $\s_1 \leq \s \leq 2$, then we have
\begin{align}	\label{maineq_1}
\log{\zeta(s)}
= \sum_{|s - \rho| \leq \delta_x}\log\l| \frac{s - \rho}{\delta_x + i(t - \gamma)} \r| 
+ \sum_{2 \leq n \leq x^{2}}\frac{\Lam_{x}(n)}{n^{s} \log{n}} + O\l( Y_{a}(\s, x, t) \r),
\end{align}
where $Y_{a}(\s, x, t)$ is defined by \eqref{def_Y}.
Moreover, if $\frac{1}{2} \leq \s \leq \s_1$, then we have
\begin{align}	\label{maineq_2}
&\log{\zeta(s)} 
= \sum_{|t - \gamma| \leq \delta_x}\log\l| \frac{s - \rho}{s_1 - \rho} \r|
+\sum_{|s_1 - \rho| \leq \delta_x}\log\l| \frac{s_1 - \rho}{\delta_x + i(t - \gamma)} \r|+\\ \nonumber
&\quad + \sum_{2 \leq n \leq x^2}\frac{\Lam_{x}(n)}{n^{s_1} \log{n}}
+ O\l((\s_1 - \s)\l( 1 + \frac{a}{\delta_x} \r)^{2}E_{a}(x, t) + Y_{a}(\s_1, x, t)\r),
\end{align}
where $E_{a}(x, t)$ is defined by \eqref{def_E}.
\end{theorem}

Theorem \ref{main_thm} immediately follows by this theorem because it is almost the same assertion as that 
of this theorem in the case of $a = \delta_x = \frac{1}{\log{x}}$, 
and we can easily obtain by evaluating the error term.
As mentioned in Section 1, the SIZDC, hence Theorem \ref{Gen_thm}, includes the unconditional case.
Actually, by taking $a = \s_A$, we obtain an assertion that is close to Theorem 2 in \cite{SCR}.

The following Lemma \ref{lem_ME} and Proposition \ref{uncon_rep} are unconditional.


\begin{lemma}	\label{lem_ME}
Let $x > 1$. 
Then for any complex number $s$ not equal to $1$ or any zero of $\zeta(s)$, we have
\begin{align*}
\sum_{n \leq x^2}\frac{\Lam_{x}(n)}{n^s}
=& -\frac{\zeta'}{\zeta}(s) + \frac{x^{2(1 - s)} - x^{1 - s}}{(1 - s)^2 \log{x}}
- \sum_{\rho}\frac{x^{2(\rho - s)} - x^{\rho - s}}{(\rho - s)^2 \log{x}}\\
 & -\sum_{k = 1}^{\infty}\frac{x^{-2(2k + s)} - x^{-2k - s}}{(2k + s)^2 \log{x}}.
\end{align*}
\end{lemma}

\begin{proof}
This assertion is Lemma 1 in \cite{SS}.
\end{proof}

\begin{proposition}	\label{uncon_rep}
Let $x$, $t$, $a$ be numbers with $t\geq 14$, $3 \leq x \leq t^2$, $0 < a \leq 1$, 
and put $\delta_x = (\log{x})^{-1}$ and $s_1 = \s_1 + it$ with $\s_1 = \frac{1}{2} + a + \delta_x$.
For $\s \geq \s_1$, 
there exists a function $\theta(s)$ such that 
the inequality $|\theta(s)| \leq 2x^{\frac{1}{2} + a - \s} \leq 2e^{-1}$ and the following formula
\begin{align*}
&\frac{\zeta'}{\zeta}(s) - \sum_{|s - \rho| \leq \delta_x}\frac{1}{s - \rho} 
= -\sum_{n \leq x^2}\frac{\Lam_{x}(n)}{n^{s}}
+ \theta(s) \l( \frac{\zeta'}{\zeta}(s_1) - \sum_{|s_1 - \rho| \leq \delta_x}\frac{1}{s_1 - \rho} \r)\\
&- \sum_{|s - \rho| \leq \delta_x}\l( \frac{x^{2(\rho - s)} - x^{\rho - s}}{(\rho - s)^2 \log{x}} + \frac{1}{s - \rho} \r) 
- \underset{|\b - \frac{1}{2}| \geq a}{\underset{|s - \rho| > \delta_x}{\sum_{\rho \in A}}}\frac{x^{2(\rho - s)} - x^{\rho - s}}{(\rho - s)^2\log{x}}\\
&+ O\l(x^{1/2 + a - \s}\log{t} + x^{1/2 + a - \s}\us{\beta \geq \s_1}{\us{|s_1 - \rho| > \delta_x}{\sum_{|t - \gamma| \leq 1}}}\Re\l(\frac{-1}{s_1 - \rho}\r) \r)
\end{align*}
hold. Here the set $A$ is defined by \eqref{def_A}.
\end{proposition}

\begin{proof}
We find that 
\begin{align*}
\frac{\zeta'}{\zeta}(s) 
= -\sum_{n \leq x^2}\frac{\Lam_{x}(n)}{n^{s}}
+ \frac{x^{2(1 - s)} - x^{1 - s}}{(1 - s)^2 \log{x}}
-\sum_{\rho}\frac{x^{2(\rho - s)} - x^{\rho - s}}{(\rho - s)^2 \log{x}}
+O\l( \frac{x^{-2 - \s}}{t^2 \log{x}} \r)
\end{align*}
by Lemma \ref{lem_ME}. We divide the sum on non-trivial zeros as
\begin{align*}
\sum_{\rho}\frac{x^{2(\rho - s)} - x^{\rho - s}}{(\rho - s)^2 \log{x}}
&= \l( \sum_{\rho \in A} + \sum_{\rho \not\in A} \r)\frac{x^{2(\rho - s)} - x^{\rho - s}}{(\rho - s)^2 \log{x}},
\end{align*}
where the set $A$ is defined by \eqref{def_A}.
We observe
\begin{align*}
\sum_{\rho \not\in A}\frac{x^{2(\rho - s)} - x^{\rho - s}}{(\rho - s)^2 \log{x}}
&=\l(\us{|t - \gamma| \leq 1}{\sum_{\rho \not\in A}} 
+ \sum_{k = 1}^{\infty}\us{k < |t - \gamma| \leq k + 1}{\sum_{\rho \not\in A}} \r)
\frac{x^{2(\rho - s)} - x^{\rho - s}}{(\rho - s)^2 \log{x}}\\
&\ll x^{1/2 - \s}\log{t} + x^{1/2 - \s}\sum_{k = 1}^{\infty}\frac{\log(t + k)}{k^{4/3}}
\ll x^{1/2 - \s}\log{t}.
\end{align*}
Therefore, if $3 \leq x \leq t^2$, then one has
\begin{align*}
\frac{\zeta'}{\zeta}(s) =
-\sum_{n \leq x^2}\frac{\Lam_{x}(n)}{n^{s}}
- \sum_{\rho \in A}\frac{x^{2(\rho - s)} - x^{\rho - s}}{(\rho - s)^2\log{x}}
+ O\l(x^{1/2 - \s}\log{t} \r)
\end{align*}
for $\s \geq \frac{1}{2}$ and $t \geq 14$.
Now we divide the second sum on the right hand side as 
\begin{align*}
\l( \underset{\l|\b - \frac{1}{2}\r| < a}{\sum_{\rho \in A}} 
+ \underset{|\b - \frac{1}{2}| \geq a}{\sum_{\rho \in A}} \r)\frac{x^{2(\rho - s)} - x^{\rho - s}}{(\rho - s)^2 \log{x}}
=: S_1 + S_2.
\end{align*}
Then we have
\begin{align*}
&|S_1| 
\leq \frac{1}{\log{x}}\underset{\l| \b - \frac{1}{2} \r| < a}{\sum_{\rho \in A}}\frac{2 x^{\beta - \s}}{(\s - \beta)^2 + (t - \gamma)^2}\\
&\leq \frac{2 x^{\frac{1}{2} + a - \s}}{\log{x}}\underset{\l| \b - \frac{1}{2} \r| < a}{\sum_{\rho \in A}}
\frac{\s_1 - \beta}{(\s_1 - \beta)^2 + (t - \gamma)^2}\frac{1}{\s_1 - \beta}
\leq 2x^{\frac{1}{2} + a - \s}
\underset{\l| \b - \frac{1}{2} \r| < a}{\sum_{\rho \in A}}\Re\l(\frac{1}{s_1 - \rho}\r).
\end{align*}
Here we use the following basic properties (cf. Section 15 \cite{DM})
\begin{gather}
\frac{\zeta'}{\zeta}(s)
= \sum_{|t - \gamma| \leq 1}\frac{1}{s - \rho}
+ O(\log{t}), \label{PEFZ}\\ \nonumber
\sum_{|t - \gamma| > 1}\l|\Re\l( \frac{1}{s - \rho} \r)\r| \ll \log{t},
\end{gather}
and so we obtain
\begin{align*}
&\underset{|\b - \frac{1}{2}| < a}{\sum_{\rho \in A}}\Re\l(\frac{1}{s_1 - \rho}\r)
= \Re\l( \frac{\zeta'}{\zeta}(s_1) - \us{|\b - \frac{1}{2}| \geq a}
{\sum_{|t - \gamma| \leq 1}}\frac{1}{s_1 - \rho} 
-\us{\rho \not\in A}{\us{|\b - \frac{1}{2}| < a}
{\sum_{|t - \gamma| \leq 1}}}\frac{1}{s_1 - \rho} \r)+ O(\log{t})\\
&\leq \Re\l( \frac{\zeta'}{\zeta}(s_1) - \sum_{|s_1 - \rho| \leq \delta_x}\frac{1}{s_1 - \rho} \r) 
+ \us{\b \geq \s_1}{\us{|s_1 - \rho| > \delta_x}{{\sum_{|t - \gamma| \leq 1}}}}\Re\l( \frac{-1}{s_1 - \rho} \r)
+O(\log{t}).
\end{align*}
Hence, we have
\begin{align*}
|S_1| 
\leq & 2x^{\frac{1}{2} + a - \s}
\Re\l( \frac{\zeta'}{\zeta}(s_1) - \sum_{|s_1 - \rho| \leq \delta_x}\frac{1}{s_1 - \rho} \r)\\
&+O\l( x^{\frac{1}{2} + a - \s}\l(\log{t} + 
\us{\beta \geq \s_1}{\us{|s_1 - \rho| > \delta_x}{\sum_{|t - \gamma| \leq 1}}}\Re\l(\frac{-1}{s_1 - \rho}\r)\r)\r).
\end{align*}
In particular, we see that $2x^{\frac{1}{2} + a - \s} \leq 2 e^{-1}$ for 
$\s \geq \frac{1}{2} + a + \delta_x$.
From the above estimates, we obtain this proposition.
\end{proof}

\begin{lemma}	\label{lem_near}
Let $t \geq 14$, $x \geq 3$, and $\delta_x = (\log{x})^{-1}$.
Assume the SIZDC-$(l, v, \Phi, \Psi)$ on $[t - \delta_x, t + \delta_x]$, and let $\frac{1}{\Psi(t / 2)} \leq a \leq 1$.
For $\s \geq \frac{1}{2} + a + \delta_x$, we have
\begin{align*}
\sum_{|s - \rho| \leq \delta_x}\l( \frac{x^{2(\rho - s)} - x^{\rho - s}}{(\rho - s)^2 \log{x}} + \frac{1}{s - \rho} \r)
\ll G_{a}(x, t) \Phi(t / 2)^{1/2 - \s + \delta_x},
\end{align*}
where the function $G_{a}(x, t)$ is defined by \eqref{def_G}. 
\end{lemma}

\begin{proof}
By the Taylor expansion, if $|s - \rho| \leq \delta_x$, then we see that
\begin{align*}
&\l|\frac{x^{2(\rho - s)} - x^{\rho - s}}{(\rho - s)^2 \log{x}} + \frac{1}{s - \rho}\r|
= \frac{1}{\log{x}}\l|\sum_{n = 2}^{\infty}\frac{(\rho - s)^{n - 2}}{n!}\l\{ (2\log{x})^n - (\log{x})^n \r\}\r|\\
&\leq \log{x}\sum_{n = 2}^{\infty}\frac{2^n}{n!}\l( |\rho - s|\log{x} \r)^{n-2}
\ll \log{x}.
\end{align*}
Hence, by the assumption SIZDC-$(l, v, \Phi, \Psi)$, we have
\begin{align*}
\sum_{|s - \rho| \leq \delta_x}\l( \frac{(xy)^{\rho - s} - x^{\rho - s}}{(\rho - s)^2 \log{y}} + \frac{1}{s - \rho} \r)
\ll \log{x} \sum_{|s- \rho| \leq \delta_x} 1
\ll G_{a}(x, t) \Phi(t / 2)^{1/2 - \s + \delta_x}.
\end{align*}
This completes the proof of this lemma.
\end{proof}

\begin{lemma}	\label{lem_zero1}
Let $t \geq 14$ and $x \geq 3$, and put $\delta_x = (\log{x})^{-1}$.
Assume the SIZDC-$(l, v, \Phi, \Psi)$ on $[t - L, t + L]$ with $L$ defined by \eqref{def_L}, 
and let $\frac{1}{\Psi(t / 2)} \leq a \leq 1$.
For $\s \geq \frac{1}{2} + a + \delta_x$, we have
\begin{align*}
\underset{|\b - \frac{1}{2}| \geq a}{\underset{|s - \rho| > \delta_x}{\sum_{\rho \in A}}}\frac{x^{2(\rho - s)} - x^{\rho - s}}{(\rho - s)^2\log{x}}
\ll G_{a}(x, t) x^{1/2 - \s} F_{a}(x, t),
\end{align*}
where the functions $F_{a}(x, t)$, $G_{a}(x, t)$ are defined by \eqref{def_F} and \eqref{def_G}, respectively.
\end{lemma}

\begin{proof}
It is clear that this lemma holds in the case $a + \frac{1}{2} \geq \s_A$. 
Therefore, we consider the case $a + \frac{1}{2} < \s_A$.
We divide the sum as
\begin{align*}
\us{|\b - \frac{1}{2}| \geq a}{\us{|s - \rho| > \delta_x}{\sum_{\rho \in A}}}\frac{x^{2(\rho - s)} - x^{\rho - s}}{(\rho - s)^2}
&= \l( \us{|\b - \frac{1}{2}| \geq a}{\us{|s - \rho| > \delta_x}{\sum_{|t - \gamma| \leq \delta_x}}}
+\us{|\b - \frac{1}{2}| \geq a}{\us{\delta_x < |t - \gamma| \leq 1}{\sum_{\rho \in A}}} + 
\us{|\b - \frac{1}{2}| \geq a}{\us{1 < |t - \gamma| \leq L}{\sum_{\rho \in A}}} \r)
\frac{x^{2(\rho - s)} - x^{\rho - s}}{(\rho - s)^2}\\
&=: S_3 + S_4 + S_5.
\end{align*}
First we consider $S_{3}$. 
Let $H_1 = \log{\Phi(t / 2)}$ and $\delta_1 = H_{1}^{-1}$.
Then, by the assumption SIZDC-$(l, v, \Phi, \Psi)$, we can find that
\begin{align*}
S_3 
&\ll (\log{x})^2 \sum_{k = 0}^{[(\s_A - a - 1/2)H_1]}\us{\frac{1}{2} + a + k\delta_1 \leq \b < \frac{1}{2} + a + (k + 1)\delta_1}
{\sum_{|t - \gamma| \leq \delta_x}}\l(x^{2(\b - \s)} + x^{\b - \s}\r)\\
&\ll (\log{x})^2 x^{1/2 - \s + a}\sum_{k = 0}^{[(\s_A - a - 1/2)H_1]}
\us{\frac{1}{2} + a + k\delta_1 \leq \b < \frac{1}{2} + a + (k + 1)\delta_1}{\sum_{|t - \gamma| \leq \delta_x}}x^{2(k + 1)\delta_1}\\
&\ll G_{a}(x, t) (\log{x}) x^{1/2 - \s} \l(\frac{x}{\Phi(t/2)} \r)^{a}
\sum_{k = 0}^{[(\s_A- a - \frac{1}{2})H_1]}\l( \frac{x^2}{\Phi(t/2)} \r)^{(k + 1)\delta_1}\\
&= G_{a}(x, t) (\log{x}) x^{1/2 - \s} F_{a}(x, t).
\end{align*}
Here the symbol $[\cdot]$ indicates the Gaussian symbol.

Next we consider $S_5$. 
Note that the inequality $t - L \geq \frac{t}{2}$ holds by the definition of $L$.
By calculating in the same manner as $S_3$, we find that
\begin{align*}
S_5 
&\ll \sum_{m = 1}^{\l[ L \r]}\underset{\frac{1}{2} + a \leq \b \leq \s_A}
{\sum_{m < |t - \gamma| \leq m + 1}}\frac{x^{2(\b - \s)} + x^{\b - \s}}{(t - \gamma)^2}\\
&\leq \sum_{m = 1}^{\l[ L \r]}\frac{1}{m^2}\sum_{k = 0}^{[(\s_A- a - \frac{1}{2})H_1]}
\underset{\frac{1}{2} + a + k\delta_1 \leq \b < \frac{1}{2} + a + (k + 1)\delta_1}
{\sum_{m < |t - \gamma| \leq m + 1}}\l(x^{2(\b - \s)} + x^{\b - \s}\r)\\
&\ll x^{1/2 - \s + a}\sum_{m = 1}^{\l[L\r]}\frac{1}{m^2}\sum_{k = 0}^{[(\s_A - a - \frac{1}{2})H_1]} x^{2(k + 1) \delta_1}
\underset{\frac{1}{2} + a + k\delta_1 \leq \b < \frac{1}{2} + a + (k + 1)\delta_1}{\sum_{m < |t - \gamma| \leq m + 1}}1\\
&\ll v(t/2) (\log{t}) x^{1 / 2 - \s}  F_{a}(x, t)
\end{align*}
by the assumption SIZDC.

Finally, we consider $S_4$. Put $H = \log{x}$. By the assumption SIZDC, we find that
\begin{align*}
|S_4| 
&\leq 2\underset{\frac{1}{2} + a \leq \b < 1}{\sum_{\delta_x < |t - \gamma| \leq 1}}\frac{x^{2(\b - \s)} + x^{\b - \s}}{(t - \gamma)^2}\\
&\leq 2\sum_{m = 1}^{[H]}\sum_{k = 0}^{[(\s_A- a - \frac{1}{2})H_1]}\underset{\frac{1}{2} + a + k\delta_1 \leq \b < \frac{1}{2} + a + (k + 1)\delta_1}
{\sum_{m\delta_x < |t - \gamma| \leq (m + 1)\delta_x}}
\frac{x^{2(\b - \s)} + x^{\b - \s}}{(t - \gamma)^2}\\
&\ll x^{1/2 - \s + a}\sum_{m = 1}^{[H]}\sum_{k = 2}^{[(\s_A- a - \frac{1}{2})H_1]}x^{2(k + 1)\delta_1}
\underset{\frac{1}{2} + a +  k\delta_1 \leq \b < \frac{1}{2} + a + (k + 1)\delta_1}{\sum_{m\delta_x < |t - \gamma| \leq (m + 1)\delta_x}}
\frac{1}{m^2 \delta_{x}^{2}}\\
&\ll G_{a}(x, t) (\log{x}) x^{1 / 2 - \s} F_{a}(x, t).
\end{align*}

From the above estimates, this lemma holds.
\end{proof}

\begin{lemma}	\label{lem_zero_real}
Let $t \geq 14$ and $x \geq 3$, and put $\delta_x = (\log{x})^{-1}$.
Assume the SIZDC-$(l, v, \Phi, \Psi)$ on $[t - 1, t + 1]$.
For $\s \geq \frac{1}{2} + \frac{1}{\Psi(t / 2)}$, we have
\begin{align*}
\underset{\beta \geq \s}{\underset{|s - \rho| > \delta_x}{\sum_{|t - \gamma| \leq 1}}}\Re\l(\frac{-1}{s - \rho} \r)
\ll G_{\s}(x, t) \l(\frac{\log{x}}{\log{\Phi(t/2)}} + 1 \r)\Phi(t/2)^{1/2 - \s},
\end{align*}
where the function $G_{\s}(x, t)$ is defined by \eqref{def_G}.
\end{lemma}

\begin{proof}
Put
\begin{align*}
\underset{\beta \geq \s}{\underset{|s - \rho| > \delta_x}{\sum_{|t - \gamma| \leq 1}}}\Re\l(\frac{-1}{s - \rho} \r)
= \l( \underset{\beta \geq \s}{\underset{|s - \rho| > \delta_x}{\sum_{|t - \gamma| \leq \delta_x}}} 
+ \underset{\beta \geq \s}{\sum_{\delta_x < |t - \gamma| \leq 1}} \r)\Re\l( \frac{-1}{s - \rho} \r)
=: S_6 + S_7.
\end{align*}
As for $S_6$, by the assumption SIZDC, we can estimate
\begin{align*}
S_6 
\leq (\log{x})\underset{\beta \geq \s}{\sum_{|t - \gamma| \leq \delta_x}}1
\ll G_{\s}(x, t)\Phi(t/2)^{1/2 - \s}.
\end{align*}
Next we consider $S_7$. Put $H = \log{x}$, $H_1 = \log{\Phi(t/2)}$, and $\delta_1 = H_{1}^{-1}$. Then we have
\begin{align*}
S_7 
&\leq \underset{\beta \geq \s}{\sum_{\delta_x < |t - \gamma| \leq 1}}\frac{\b - \s}{(t - \gamma)^2}
\leq \sum_{k = 1}^{[H]}\underset{\b \geq \s}{\sum_{k \delta_x < |t - \gamma| \leq (k + 1)\delta_x}}\frac{\b - \s}{\delta_{x}^{2} k^2}\\
&\ll (\log{x})^2 \sum_{k = 1}^{[H]}\frac{1}{k^2}\sum_{m = 0}^{[(1 - \s)H_1]}
\underset{\s + m\delta_1 \leq \b < \s + (m + 1)\delta_1}{\sum_{k \delta_x < |t - \gamma| \leq (k + 1)\delta_x}}(\b - \s)\\
&\ll (\log{x})^2 \sum_{k = 1}^{[H]}\frac{1}{k^2}\sum_{m = 0}^{[(1 - \s)H_1]}
\underset{\s + m\delta_1 \leq \b < \s + (m + 1)\delta_1}{\sum_{k \delta_x < |t - \gamma| \leq (k + 1)\delta_x}}(m + 1)\delta_1\\
&\ll G_{\s}(x, t)\delta_1 (\log{x})\sum_{k = 1}^{[H]}\frac{1}{k^2}\sum_{m = 0}^{[(1 - \s)H_1]}(m + 1) \Phi(t/2)^{1 / 2 - (\s + m\delta_1)}\\
&\ll G_{\s}(x, t) \frac{\log{x}}{\log{\Phi(t/2)}} \Phi(t/2)^{1/2 - \s}\sum_{m = 0}^{[(1 - \s)H_1]}(m + 1)e^{-m}\\
&\ll G_{\s}(x, t) \frac{\log{x}}{\log{\Phi(t/2)}} \Phi(t/2)^{1/2 - \s}.
\end{align*}
Hence, we complete the proof of this lemma.
\end{proof}

Now we can obtain the following proposition by the above consequences.

\begin{proposition}	\label{fun_prop1}
Let $t \geq 14$, $3 \leq x \leq t^2$, and put $\delta_x = (\log{x})^{-1}$.
Assume the SIZDC-$(l, v, \Phi, \Psi)$ on $[t - L, t + L]$ with $L$ defined by \eqref{def_L}, and let $\frac{1}{\Psi(t / 2)} \leq a \leq 1$ and $s_1 = \s_1 + it$ with $\s_1 = \frac{1}{2} + a + \delta_x$.
For $\s \geq \s_1$, we have
\begin{align*}
&\frac{\zeta'}{\zeta}(s) - \sum_{|s - \rho| \leq \delta_x}\frac{1}{s - \rho} + \sum_{n \leq x^2}\frac{\Lam_{x}(n)}{n^s}\\
&\ll x^{1/2 + a - \s}\l| \sum_{n \leq x^2}\frac{\Lam_{x}(n)}{n^{s_1}} \r| + x^{1 / 2 + a - \s}\log{t} + G_{a}(x, t) \times\\ 
& \; \; \;  \times\l( x^{\frac{1}{2} - \s}F_{a}(x, t) + x^{\frac{1}{2} - \s}
\l(1 + \frac{\Phi(t/2)^{-\delta_x}\log{x}}{\log{\Phi(t/2)}}\r) \l(\frac{x}{\Phi(t/2)}\r)^a + \Phi(t/2)^{\frac{1}{2} - \s + \delta_x} \r),
\end{align*}
where the function $F_{a}(x, t)$ is defined by \eqref{def_F}.
In particular, we have
\begin{align}	\label{rep_s_1}
\frac{\zeta'}{\zeta}(s_1) - \sum_{|s_1 - \rho| \leq \delta_x}\frac{1}{s_1 - \rho} + \sum_{n \leq x^2}\frac{\Lam_{x}(n)}{n^{s_1}}
\ll E_{a}(x, t).
\end{align}
Here the function $E_{a}(x, t)$ is defined by \eqref{def_E}.
\end{proposition}

\begin{proof}
By Proposition \ref{uncon_rep}, Lemmas \ref{lem_near}, \ref{lem_zero1}, and \ref{lem_zero_real}, we can find that
\begin{align*}
\frac{\zeta'}{\zeta}(s_1) - \sum_{|s_1 - \rho| \leq \delta_x}\frac{1}{s_1 - \rho}
\ll E_{a}(x, t).
\end{align*}
Moreover, applying this estimate to the right hand side of Proposition \ref{uncon_rep} and using lemmas again, 
we obtain this proposition.
\end{proof}

\begin{lemma}	\label{lem_near_critical}
Let $t \geq 14$, $3 \leq x \leq t^2$, and put $\delta_x = (\log{x})^{-1}$.
Assume the SIZDC-$(l, v, \Phi, \Psi)$ on $[t - L, t + L]$ with $L$ defined by \eqref{def_L}, 
and let $\frac{1}{\Psi(t / 2)} \leq a \leq 1$.
Then, we have
\begin{align*}
N\l(t + a + \delta_x \r) - N(t)
\ll (a + \delta_x) E_{a}(x, t),
\end{align*}
where $E_{a}(x, t)$ is defined by \eqref{def_E}.
\end{lemma}

\begin{proof}
Let $s_{1} = \s_{1} + it$ with $\s_{1} = \frac{1}{2} + a + \delta_x$.
By equation \eqref{PEFZ} and inequality \eqref{rep_s_1}, we have
\begin{align}	\label{Res_1}
\underset{|s_1 - \rho| > \delta_x}{\sum_{|t - \gamma| \leq 1}}\Re\l(\frac{1}{s_1 - \rho}\r)
\ll E_{a}(x, t).
\end{align}
In the following, we consider the sum on the left hand side of this inequality.

First, we divide the sum as
\begin{align*}
&\l( \underset{\b \geq \s_1}{\underset{|s_1 - \rho| > \delta_x}{\sum_{|t - \gamma| \leq 1}}}
+ \underset{|\b - \frac{1}{2}| < a + \delta_x}{\underset{|s_1 - \rho| > \delta_x}{\sum_{|t - \gamma| \leq 1}}}
+ \underset{0 < \b \leq 1 - \s_1}{\sum_{|t - \gamma| \leq \delta_x}}
+\underset{0 < \b \leq 1 - \s_1}{\sum_{\delta_x < |t - \gamma| \leq 1}} \r)\Re\l(\frac{1}{s_1 - \rho}\r)\\
&=: S_{8} + S_{9} + S_{10} + S_{11},
\end{align*}
say.
By Lemma \ref{lem_zero_real}, we obtain $S_{8} \ll G_{a}(x, t)\l(1 + \frac{\log{x}}{\log{\Phi(t / 2)}}\r)\Phi(t/2)^{-a - \delta_x}$.
We also obtain $S_{10} \ll G_{a}(x, t) \Phi(t / 2)^{-a - \delta_x}$ by the assumption SIZDC.

Next we consider $S_{11}$. It is clear that $S_{11} = 0$ when $a > \s_A - \frac{1}{2}$. 
Hence, we consider the case of $a \leq \s_A - \frac{1}{2}$.
Put $H = \log{x}$, $H_1 = \log{\Phi(t / 2)}$, and $\delta_1 = {H_1}^{-1}$.
Then we have
\begin{align*}
|S_{11}|
&= \underset{0 < \b \leq 1 - \s_1}{\sum_{\delta_x < |t - \gamma| \leq 1}}\frac{\s_1 - \b}{(\s_1 - \b)^2 + (t - \gamma)^2}\\
&\leq \sum_{m = 0}^{[(1 - \s_1)H_1]}\sum_{k = 1}^{[H]}\underset{1 - \s_1 - (m + 1)\delta_1 < \b \leq 1 - \s_1 - m \delta_1}
{\sum_{k\delta_x < |t - \gamma| \leq (k + 1)\delta_x}}\frac{\s_1 - \b}{(\s_1 - \b)^2 + (t - \gamma)^2}\\
&\leq \sum_{m = 0}^{[(1 - \s_1)H_1]} \sum_{k = 1}^{[H]}\frac{2a + 2\delta_x + (m + 1)\delta_1}{(2a + 2\delta_x + m\delta_1) + k^2\delta_x^2}
\underset{1 - \s_1 - (m + 1)\delta_1 < \b \leq 1 - \s_1 - m\delta_1}{\sum_{k\delta_x < |t - \gamma| \leq (k + 1)\delta_x}}1\\
&\ll G_{a}(x, t) \Phi(t/2)^{-a - \delta_x} \sum_{m = 0}^{[H_1]}\frac{2a + 2\delta_x + (m + 1)\delta_1}{2a + 2\delta_x + m\delta_1}e^{-m}\\
&\ll G_{a}(x, t) \l(1 + \frac{\log{x}}{\log{\Phi(t / 2)}}\r)\Phi(t / 2)^{-a - \delta_x}.
\end{align*}
By these estimates and \eqref{Res_1}, we obtain
\begin{align*}
S_{9} 
\ll E_{a}(x, t).
\end{align*}
Moreover, $S_{9}$ can be estimated by
\begin{align*}
S_{9}
\geq \underset{|\b - \frac{1}{2}| \leq a}{\sum_{|t - \gamma| \leq 1}}\frac{\s_1 - \b}{(\s_1 - \b)^2 + (t - \gamma)^2}
\gg \frac{1}{a + \delta_x} \underset{|\b - \frac{1}{2}| \leq a}{\sum_{|t - \gamma| \leq a+ \delta_x}}1.
\end{align*}
Hence, we have
\begin{align*}
\underset{|\b - \frac{1}{2}| \leq a}{\sum_{|t - \gamma| \leq a+ \delta_x}}1
\ll (a + \delta_x) E_{a}(x, t).
\end{align*}
By combining this inequality and the assumption SIZDC-$(l, v, \Phi, \Psi)$, we obtain this lemma.
\end{proof}

\begin{lemma}	\label{rep_near_c_Z}
Let $t \geq 14$, $3 \leq x \leq t^2$, and put $\delta_x = (\log{x})^{-1}$.
Assume the SIZDC-$(l, v, \Phi, \Psi)$ on $[t - L, t + L]$ with $L$ defined by \eqref{def_L}, 
and let $\frac{1}{\Psi(t / 2)} \leq a \leq 1$ and $s_{1} = \s_{1} + it$ 
with $\s_{1} = \frac{1}{2} + a + \delta_x$.
For $\frac{1}{2} \leq \s \leq \s_1$, we have
\begin{align*}
\frac{\zeta'}{\zeta}(s) 
= \sum_{|t - \gamma| \leq \delta_x}\frac{1}{s - \rho} - \sum_{n \leq x^2}\frac{\Lam_{x}(n)}{n^{s_1}} + O\l(\l( 1 + \frac{a}{\delta_x} \r)^{2} E_{a}(x, t)\r),
\end{align*}
where $E_{x}(t)$ is defined by \eqref{def_E}.
\end{lemma}

\begin{proof}
We divide the sum for non-trivial zeros of equation \eqref{PEFZ} as
\begin{align*}
\frac{\zeta'}{\zeta}(\s + it) 
&= \l(\sum_{|t - \gamma| \leq \delta_x} + {\sum_{\delta_x < |t - \gamma| \leq 1}} \r)\frac{1}{s - \rho} + O(\log{t})\\
&=: Q_1 + Q_2 + O(\log{t}).
\end{align*}
Therefore, to complete the proof it suffices to show
\begin{align}
Q_2 = -\sum_{n \leq x^2}\frac{\Lam_{x}(n)}{n^{s_x}} + O\l(\l( 1 + \frac{a}{\delta_x} \r)^{2} E_{a}(x, t)\r).
\end{align}
As a preparation, we first show 
\begin{align}	\label{lem_for_rep_lem}
\sum_{\delta_{x} < |t - \gamma| \leq 1}\frac{1}{s_1 - \rho} 
= -\sum_{n \leq x^2}\frac{\Lam_{x}(n)}{n^{s_1}} + O\l(\l( 1 + \frac{a}{\delta_x} \r) 
E_{a}(x, t)\r).
\end{align}
By \eqref{PEFZ} and \eqref{rep_s_1}, we have
\begin{align*}
\us{|s_1 - \rho| > \delta_x}{\sum_{|t - \gamma| \leq 1}}\frac{1}{s_1 - \rho}
= -\sum_{n \leq x^2}\frac{\Lam_{x}(n)}{n^{s_1}} + O(E_{a}(x, t)),
\end{align*}
and by Lemma \ref{lem_near_critical}, we have
\begin{align*}
\l|\us{|s_1 - \rho| > \delta_x}{\sum_{|t - \gamma| \leq \delta_x}}\frac{1}{s_1 - \rho}\r|
\leq \delta_x^{-1}\sum_{|t - \gamma| \leq \delta_x}1 
\ll \l(1 + \frac{a}{\delta_x}\r) E_{a}(x, t).
\end{align*}
Hence, we obtain \eqref{lem_for_rep_lem}.

By using \eqref{lem_for_rep_lem}, we have
\begin{align*}
Q_2 
= -\sum_{n \leq x^2}\frac{\Lam_{x}(n)}{n^{s_1}}
+ \sum_{\delta_x < |t - \gamma| \leq 1}\l( \frac{1}{s - \rho} - \frac{1}{s_1 - \rho} \r) + O\l( \l( 1 + \frac{a}{\delta_x} \r) E_{a}(x, t) \r).
\end{align*}
Put $H = \log{x}$. Now, applying Lemma \ref{lem_near_critical} to the second term on the right hand side 
of the above, we find that
\begin{align*}
&\sum_{\delta_x < |t - \gamma| \leq 1}\l( \frac{1}{s - \rho} - \frac{1}{s_1 - \rho} \r)
\ll (\s_1 - \s)\sum_{\delta_x < |t - \gamma| \leq 1}\frac{1}{|t - \gamma|^2}\\
&\ll (a + \delta_x) \sum_{k = 1}^{[H]}\sum_{k \delta_x < |t - \gamma| \leq (k + 1)\delta_x}\frac{1}{k^2 \delta_{x}^{2}}
\ll \l(\frac{a + \delta_x}{\delta_{x}^2}\r)\sum_{k = 1}^{[H]}\frac{1}{k^2}\sum_{k \delta_x < |t - \gamma| \leq (k + 1) \delta_x}1\\
&\ll \l( 1 + \frac{a}{\delta_x} \r)^{2} E_{a}(x, t).
\end{align*}
Hence, we obtain this lemma.
\end{proof}

\section{\textbf{Proof of Theorem \ref{Gen_thm}}}

\begin{proof}[Proof of Theorem \ref{Gen_thm}]
Let $\frac{1}{2} \leq \s \leq 2$ and 
$t \geq 14$ be not the ordinate of zeros of the Riemann zeta-function, and $3 \leq x \leq t^2$.
First we show the theorem in the case of $\frac{1}{2} + a + \delta_x =: \s_1 \leq \s \leq 2$.
We have
\begin{align}	\label{rep_log_Z}
\log{\zeta(\s + it)}
= \int_{2}^{\s} \frac{\zeta'}{\zeta}(\a + it) d\a + O(1).
\end{align}
Now, by using Proposition \ref{fun_prop1}, we have
\begin{align*}
\log{\zeta(\s + it)}
= &\int_{2}^{\s} \l(\sum_{|(\a - \b) - i(t - \gamma)| \leq \delta_x}\frac{1}{(\a - \b) + i(t - \gamma)}\r)d\a
+\sum_{2 \leq n \leq x^2}\frac{\Lam_{x}(n)}{n^{\s + it} \log{n}}\\
&+ O\l( Y_{a}(\s, x, t) \r),
\end{align*}
where $Y_{a}(\s, x, t)$ is defined by \eqref{def_Y}.
Here, by the assumption SIZDC-$(l, v, \Phi, \Psi)$, we have
\begin{align*}
&\sum_{|(\a - \b) - i(t - \gamma)| \leq \delta_x}\frac{1}{(\a - \b) + i(t - \gamma)}\\
&= \us{|\a - \b| \leq \delta_x}{\sum_{|t - \gamma| \leq \delta_x}}\frac{1}{(\a - \b) + i(t - \gamma)}
+ O\l( G_{a}(x, t) \Phi(t / 2)^{1/2 - \a} \r)
\end{align*}
for $\s_1 \leq \a \leq 2$. Therefore, we find that
\begin{align*}
&\int_{2}^{\s} \l(\sum_{|(\a - \b) - i(t - \gamma)| \leq \delta_x}\frac{1}{(\a - \b) + i(t - \gamma)}\r)d\a\\
&= \int_{2}^{\s}\us{|\a - \b| \leq \delta_x}{\sum_{|t - \gamma| \leq \delta_x}}\frac{d\a}{(\a - \b) + i(t - \gamma)} + O\l( Y_{a}(\s, x, t) \r)\\
&= -\us{\s - \delta_x \leq \b < 1}{\sum_{|t - \gamma| \leq \delta_x}}\int_{\max\l\{ \b - \delta_x, \s \r\}}^{\b + \delta_x}\frac{d\a}{(\a - \b) + i(t - \gamma)} 
+ O\l( Y_{a}(\s, x, t) \r)\\
&= -\us{\s - \delta_x \leq \b \leq \s + \delta_x}{\sum_{|t - \gamma| \leq \delta_x}}\int_{\s}^{\b + \delta_x}\frac{d\a}{(\a - \b) + i(t - \gamma)}
-\us{\s + \delta_x < \b < 1}{\sum_{|t - \gamma| \leq \delta_x}}\int_{\b - \delta_x}^{\b + \delta_x}\frac{d\a}{(\a - \b) + i(t - \gamma)}\\
& \qquad + O\l( Y_{a}(\s, x, t) \r).
\end{align*}
By the assumption SIZDC-$(l, v, \Phi, \Psi)$, we can obtain 
\begin{align*}
&\us{\s - \delta_x \leq \b \leq \s + \delta_x}{\sum_{|t - \gamma| \leq \delta_x}}\int_{\s}^{\b + \delta_x}\frac{d\a}{(\a - \b) + i(t - \gamma)}\\
&= -\us{\s - \delta_x \leq \b \leq \s + \delta_x}{\sum_{|t - \gamma| \leq \delta_x}}
\log\l| \frac{s - \rho}{\delta_x + i(t - \gamma)} \r| + O\l( Y_{a}(\s, x, t) \r)\\
&= -\sum_{|s - \rho| \leq \delta_x}
\log\l| \frac{s - \rho}{\delta_x + i(t - \gamma)} \r| + O\l( Y_{a}(\s, x, t) \r)
\end{align*}
and 
\begin{align*}
\us{\s + \delta_x < \b < 1}{\sum_{|t - \gamma| \leq \delta_x}}\int_{\b - \delta_x}^{\b + \delta_x}\frac{d\a}{(\a - \b) + i(t - \gamma)}
\ll Y_{a}(\s, x, t).  
\end{align*}
Hence, for $\s_1 \leq \s \leq 2$, we have
\begin{align}	\label{logz_1}
\log{\zeta(s)}
=\sum_{|s - \rho| \leq \delta_x}
\log\l| \frac{s - \rho}{\delta_x + i(t - \gamma)} \r|
+ \sum_{2 \leq n \leq x^{2}}\frac{\Lam_{x}(n)}{n^{s} \log{n}}
+ O\l( Y_{a}(\s, x, t) \r).
\end{align}

Next, we consider the case $\frac{1}{2} \leq \s \leq \s_1$.
Now by \eqref{logz_1}, we have
\begin{align*}
\log{\zeta(s)}
=& \int_{\s_1}^{\s}\frac{\zeta'}{\zeta}(\a + it)d\a + \log{\zeta(s_1)}\\
=& \int_{\s_1}^{\s}\frac{\zeta'}{\zeta}(\a + it)d\a
+\sum_{|s_1 - \rho| \leq \delta_x}
\log\l| \frac{s_1 - \rho}{\delta_x + i(t - \gamma)} \r|\\
&+ \sum_{2 \leq n \leq x^{2}}\frac{\Lam_{x}(n)}{n^{s_1} \log{n}}
+ O\l( Y_{a}(\s_1, x, t) \r),
\end{align*}
and by Lemma \ref{lem_near_critical} and Lemma \ref{rep_near_c_Z}, we have
\begin{align*}
\int_{\s_1}^{\s}\frac{\zeta'}{\zeta}(\a + it)d\a
= &\int_{\s_1}^{\s}\sum_{|t - \gamma| \leq \delta_x}\frac{d\a}{(\a - \b) + i(t - \gamma)}\\
&+ O\l((a + \delta_x) \l|\sum_{n \leq x^2}\frac{\Lam_{x}(n)}{n^{s_1}}\r| 
+ (\s_1 - \s)\l( 1 + \frac{a}{\delta_x} \r)^{2}E_{a}(x, t)\r)\\
= & \sum_{|t - \gamma| \leq \delta_x}
\log\l| \frac{s - \rho}{s_1 - \rho} \r|
+ O\l((\s_1 - \s)\l( 1 + \frac{a}{\delta_x} \r)^{2}E_{a}(x, t)\r).
\end{align*}
Hence, again applying Lemma \ref{lem_near_critical}, we can find that
\begin{align*}
\log{\zeta(\s + it)} 
=&\sum_{|t - \gamma| \leq \delta_x}
\log\l| \frac{s - \rho}{s_1 - \rho} \r|
+\sum_{|s_1 - \rho| \leq \delta_x}\log\l| \frac{s_1 - \rho}{\delta_x + i(t - \gamma)} \r|\\
&+ \sum_{2 \leq n \leq x^2}\frac{\Lam_{x}(n)}{n^{s_1} \log{n}}
+ O\l((\s_1 - \s)\l( 1 + \frac{a}{\delta_x} \r)^{2}E_{a}(x, t) + Y_{a}(x, t) \r).
\end{align*}
From the above calculations, we obtain Theorem \ref{Gen_thm}.
\end{proof}

\begin{acknowledgment*}
The author expresses his gratitude to Professors Kohji Matsumoto, Yoonbok Lee, and Masatoshi Suzuki for their helpful comments.
\end{acknowledgment*}





\end{document}